\documentclass[10pt]{amsart}

\usepackage{amsmath, amsthm}
\usepackage{amsbsy}
\usepackage{eucal}

\usepackage{amssymb}
\usepackage{amscd}
\usepackage{latexsym}
\usepackage{epsfig}
\usepackage{graphicx}
\usepackage{amsfonts}
\usepackage{psfrag}
\usepackage{color}
\usepackage{curves}

\usepackage{tikz,subcaption,mathtools,tikz-cd}
\usepackage{booktabs}
\usetikzlibrary{decorations.markings}
\usepackage{array,calc}

\usepackage{url}
\usepackage{cite}

\usepackage{hyperref}

\input xy

\usepackage{caption}

\newtheorem*{theorem*}{Theorem}

\newtheorem{theorem}{Theorem}[section]
\newtheorem{lemma}[theorem]{Lemma}
\newtheorem{proposition}[theorem]{Proposition}
\newtheorem{corollary}[theorem]{Corollary}

\theoremstyle{definition}
\newtheorem{example}[theorem]{Example}
\newtheorem{definition}[theorem]{Definition}
\newtheorem{remark}[theorem]{Remark}

\numberwithin{equation}{section}
\numberwithin{figure}{section}

\usepackage{caption}
\usepackage[labelformat=simple]{subcaption}

\DeclareMathAlphabet{\mathpzc}{OT1}{pzc}{m}{it}

\newcommand{\C}{{\mathbb{C}}}

\newcommand{\Z}{{\mathbb{Z}}}

\newcommand{\R}{{\mathbb{R}}}

\renewcommand{\c}{{\mathpzc{c}}}

\newcommand{\mf}{\mathbf}
\newcommand{\mfc}{\mf{c}}
\newcommand{\mfl}{\boldsymbol{\ell}}
\newcommand{\mfm}{\mf{m}}

\newcommand{\al}{\alpha}

\newcommand{\de}{\delta}
\newcommand{\lee}{\langle}
\newcommand{\ree}{\rangle}

\newcommand{\defi}[1]{\textit{#1}}

\DeclareMathOperator{\SO}{SO}

\DeclareMathOperator{\SL}{SL}

\definecolor{gold}{rgb}{0.85,.66,0}
\definecolor{cherry}{rgb}{0.9,.1,.2}
\definecolor{burgundy}{rgb}{0.8,.2,.2}
\definecolor{orangered}{rgb}{0.85,.3,0}
\definecolor{orange}{rgb}{0.85,.4,0}
\definecolor{olive}{rgb}{.45,.4,0}
\definecolor{lime}{rgb}{.6,.9,0}
\definecolor{green}{rgb}{.2,.7,0}
\definecolor{grey}{rgb}{.4,.4,.2}
\definecolor{brown}{rgb}{.4,.3,.1}

\begin{document}

\title[\tiny{Grossberg--Karshon twisted cubes and hesitant jumping walk avoidance}]{Grossberg--Karshon twisted cubes \\ and hesitant jumping walk avoidance} 

\date{\today}

\author[Eunjeong Lee]{Eunjeong Lee}
\address[E. Lee]{Center for Geometry and Physics, Institute for Basic Science (IBS), Pohang 37673, Republic of Korea}
\email{eunjeong.lee@ibs.re.kr}


\date{\today}



\begin{abstract}
	Let $G$ be a complex simply-laced semisimple algebraic group of rank $r$ and $B$ a Borel subgroup. Let $\mathbf i \in [r]^n$ be a word  and let $\mfl = (\ell_1,\dots,\ell_n)$ be a sequence of non-negative integers. Grossberg and Karshon introduced a virtual lattice polytope associated to $\mathbf i$ and $\mfl$ called a \defi{twisted cube}, whose lattice points encode the character of a $B$-representation.
	More precisely, lattice points in the twisted cube, counted with sign according to a certain density function, yields the character of the generalized Demazure module determined by $\mathbf i$ and~$\mfl$.
	In recent work, the author and Harada described precisely when the Grossberg--Karshon twisted cube is untwisted, i.e., the twisted cube is a closed convex polytope, in the situation when the integer sequence $\mfl$ comes from a weight $\lambda$ of~$G$. However, not every integer sequence $\mfl$ comes from a weight of $G$. In the present paper, we interpret untwistedness of Grossberg--Karshon twisted cubes associated to any word~$\mathbf i$ and any integer sequence~$\mfl$ using the combinatorics of $\mathbf i$ and $\mfl$. Indeed, we prove that the Grossberg--Karshon twisted cube is untwisted precisely when $\mathbf i$ is  \defi{hesitant-jumping-$\mfl$-walk-avoiding}. 
%
 \end{abstract}

\maketitle



\section*{Introduction} 

Let $G$ be a complex semisimple algebraic group of rank $r$ and $B$ a Borel subgroup.
Formulating a combinatorial model for a basis of a representation provides a fruitful connection between representation theory and algebraic geometry as exhibited by the theory of crystal bases and string polytopes. 
Kaveh~\cite{Kaveh} show that the string polytopes can be obtained as Newton--Okounkov bodies of the flag variety~$G/B$, and this association is extended to the generalized string polytopes and Bott--Samelson varieties by  Fujita~\cite{Fujita}. Furthermore, using the result of Anderson~\cite{Anderson}, for this case there is a  toric degeneration of Bott--Samelson variety to a toric variety whose Newton polytope is the generalized string polytope. 

On the other hand, Grossberg and Karshon \cite{Grossberg-Karshon} also
constructed  one-parameter family of complex structures on Bott--Samelson varieties which makes the Bott--Samelson varieties into toric varieties, called Bott manifolds, and consequently obtained a Demazure-type character
formula which can be interpreted combinatorially in terms of 
\emph{twisted cubes}. This degeneration of complex structures can be interpreted as the toric degeneration of Bott--Samelson variety to a Bott manifold by Pasquier~\cite{Pasquier}. Indeed, there is a flat family $\mathfrak{X}$ over $\C$ such that $\mathfrak{X}(t)$ is isomorphic to the Bott--Samelson variety for all $t  \in \C \setminus \{0\}$ and $\mathfrak{X}(0)$ is a Bott manifold. This connection is generalized to flag Bott--Samelson varieties and flag Bott manifolds in~\cite{FLS}.

These twisted cubes are combinatorially much
simpler than generalized string polytopes but they are not \textit{actual}
polytopes in the sense that they may not be convex nor closed and the intersection of faces may not be a face (cf. \cite[\S 2.5 and Figure~1 therein]{Grossberg-Karshon} and Figure~\ref{figure:first example}). More precisely, a Grossberg--Karshon
twisted cube is a pair $(C = C(\mfc, \mfl), \rho)$, where $C$ is a subset of $\R^n$ and
$\rho$ is a density function whose support is  $C$, taking
values in $\{\pm 1\}$. The defining parameters $\mfc = \{c_{jk}\}_{1 \leq j < k \leq n}$ and $\mfl = (\ell_1, \ldots, \ell_n)$ are fixed constants 
with $c_{jk} \in \Z$ and $\ell_j \in \R$.

The main result of this paper concerns twisted cubes obtained from representation-theoretic data. Indeed, we consider a (not necessarily reduced) word decomposition $\mathbf i = (i_1,\dots,i_n) \in [r]^n$ of an element $s_{i_1} \cdots s_{i_n}$ in the Weyl group $W$ of $G$ and non-negative integers $\mfm = (m_1,\dots,m_n)$. Here, $[r] \coloneqq \{1,\dots,r\}$.
In this situation, the sequence $\mathbf i$ and non-negative integers~$\mfm$ define a Bott--Samelson variety $Z_{\mf i}$ and the line bundle $\mathcal L_{\mathbf i, \mfm}$ on it. Moreover, the associated Grossberg--Karshon twisted cube $(C(\mfc(\mathbf i), \mfl(\mathbf i, \mfm)),\rho)$ encodes the character of $B$-representation space $H^0(Z_{\bf i}, \mathcal L_{\mathbf i, \mfm})$ of holomorphic sections. Here, the integers $\mfc(\mathbf i)$ and $\mfl(\mathbf i, \mfm)$ are determined by $\mathbf i$ and $\mfm$ (see Section~\ref{sec:walks} for more details).

In this paper, we present a necessary and sufficient conditions on $\bf i$ and $\mfl$ such that the associated Grossberg--Karshon twisted cube is untwisted (see Definition~\ref{definition:untwisted}), i.e., $C(\mfc, \mfl)$ is a closed convex polytope and the density function is equal to $1$ on $C(\mfc, \mfl)$, so that the Grossberg--Karshon character formula is a purely combinatorial positive formula. In other words, there is no minus sign in the formula. 

In order to introduce our result, we prepare some terminology (see Section~\ref{sec:walks} for precise definitions). We say a word $\mathbf i = (i_1,\dots,i_n)$ is a \defi{jumping walk} if for each $ 1 \leq j \leq n$, the set $\{i_1,\dots,i_{j-1}\}$ and an element $i_j$ are adjacent in the Dynkin diagram, i.e., the distance $d(\{i_1,\dots,i_{j-1}\}, i_j) = \min\{ d(i_k,i_j) \mid k=1,\dots,j-1\}$ is one. Here, $d(a,b)$ is the distance of two nodes $a$ and $b$ on the Dynkin diagram. Therefore, to make a jumping walk, 
one can \textit{jump} for the next step, but cannot away far. For example, in type $A_5$
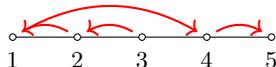
\begin{center}
	\begin{tikzpicture}[scale=.43]
\foreach \x in {0,...,4}
\draw[xshift=\x cm] (\x cm,0) circle (.1cm) node (\x) {};
\foreach \y in {0.05,...,3.05}
\draw[xshift=\y cm] (\y cm,0) -- +(1.8 cm,0);
\node [below] at (0,-.2) {\small{$1$}};
\node [below] at (2,-.2) {\small{$2$}};
\node [below] at (4,-.2) {\small{$3$}};
\node [below] at (6,-.2) {\small{$4$}};
\node [below] at (8,-.2) {\small{$5$}};

\path[->, red, thick] (2) edge [bend right] (1) {}
(1) edge [bend right] (0) {}
(0) edge [bend left] (3) {}
(3) edge [bend left] (4) {};

\end{tikzpicture}
\end{center}
the word $\mathbf i = (1,3,2,4)$ is not a jumping walk since $d(1,3) =2$, but $\mathbf i = (3,2,1,4,5)$ is a jumping walk. In the above diagram, one can see the jumping walk $(3,2,1,4,5)$. The word $\mathbf i = (i_1, i_2,\dots,i_n)$ is a \textit{hesitant jumping $\mfl$-walk} if $i_1 = i_2$, the subword $(i_2,\dots,i_n)$ is a jumping walk, and the integers satisfies an inequality $\ell_1 - \ell_2 < \ell_2 + \cdots + \ell_n$. Finally, we say that $\mathbf i$ is hesitant-jumping-$\mfl$-walk-avoiding if there is no subword which is a hesitant jumping $\mfl$-walk. Now we state our main theorem.
\begin{theorem*}[{Theorem~\ref{theorem:main}}]
	Let $G$ be a  complex simply-laced semisimple algebraic group of rank $r$.
	Let $\mathbf i = (i_1,\dots,i_n) \in [r]^n$ and $\mfl = (\ell_1,\dots,\ell_n) \in \Z_{\geq 0}^n$. 
	Then the corresponding Grossberg--Karshon twisted cube is untwisted if and only if $\mathbf i$ is hesitant-jumping-$\mfl$-walk-avoiding. 
\end{theorem*}

We note that if we consider the situation when the line bundle $\mathcal L_{\mathbf i, \mfm}$ comes from the line bundle $\mathcal L_{\lambda}$ of $G/B$, the untwistedness of the corresponding Grossberg--Karshon twisted cube can be detected using \defi{hesitant $\lambda$-walk avoidance} by Harada and the author~\cite{HaradaLee}. The Picard number of the Bott--Samelson variety $Z_{\mathbf i}$ is $n$, the length of the sequence $\mathbf i$, but on the other hand, that of the flag variety $G/B$ is $r$, the rank of the group $G$. Accordingly, not every line bundle over the Bott--Samelson variety comes from the line bundle $\mathcal L_{\lambda}$ of $G/B$ usually. Since our main result can be applied to any line bundles over $Z_{\mathbf i}$, this result is more powerful than the previous one~\cite{HaradaLee} (see Remark~\ref{rmk_gDemazure_module} and Corollary~\ref{corollary_HL}).

Additionally, for given $\mathbf i \in [r]^n$ and $\mfm$, with an appropriate choice of a valuation $\nu$ on the function field $\C(Z_{\mathbf i})$, Harada and Yang~\cite{Harada-Yang:2014a} construct a Newton--Okounkov body $\Delta = \Delta(Z_{\mathbf i}, \mathcal L_{\mathbf i, \mfm}, \nu)$ of the Bott--Samelson variety $Z_{\mathbf i}$. They proved that when the twisted cube $(C(\mfc(\mathbf i), \mfl(\mathbf i, \mfm)),\rho)$  is untwisted, then the Newton--Okounkov body~$\Delta$ and the twisted cube $C(\mfc(\mathbf i), \mfl(\mathbf i, \mfm))$ are the same (up to certain coordinate changes). Our result presents a sufficient condition on $\mathbf i$ and $\mfm$ so that the Newton--Okounkov body $\Delta$ coincides with the twisted cube.

This paper is organized as follows. In Section~\ref{sec:background}, we
recall the necessary definitions and establish terminology and
notation. In Section~\ref{sec:walks}, we introduce the notions of
jumping walks, hesitant jumping walks, and hesitant-jumping-$\mfl$-walk-avoidance. Using this
terminology we then make the statement of our main result, which is
that untwistedness is equivalent to hesitant-jumping-$\mfl$-walk-avoidance. 
The proof of the main result occupies Section~\ref{sec:proof}.

\medskip

\noindent \textbf{Acknowledgements.} 
The author thanks Professor Megumi Harada for bring her this question, and thanks Professor  Dong Youp Suh for his supports throughout the project.
The author was supported by Basic Science 
Research Program through the National Research Foundation of Korea (NRF) 
funded by the Ministry of  Science, ICT \& Future Planning 
(No. 2016R1A2B4010823) and IBS-R003-D1.

\section{Background on Grossberg--Karshon twisted cubes}\label{sec:background} 
We begin by recalling the definition of \textit{twisted cubes} introduced
by Grossberg and Karshon \cite[\S 2.5]{Grossberg-Karshon}. 
Let $n$ be a fixed positive integer.  A twisted
cube is defined to be a pair $(C(\mfc, \mfl), \rho)$ where $C(\mfc, \mfl)$
is a subset of $\R^n$ and $\rho: \R^n \to \R$ is a density function
with support equals to $C(\mfc, \mfl)$. Here, $\mfc =
\{c_{jk}\}_{1 \leq j < k \leq n}$ and
$\mfl = (\ell_1, \ell_2, \ldots, \ell_n)$ are fixed integers. 
In order to simplify the notation in what follows, we define
the following functions on $\R^n$: 
\begin{equation}\label{eq:def-A} 
\begin{split} 
A_n(x) = A_n(x_1, \ldots, x_n) & = \ell_n, \\
A_j(x) = A_j(x_1, \ldots, x_n) & = \ell_j - \sum_{k > j} c_{jk} x_k
\textup{ for all } 1 \leq j \leq n-1. 
\end{split} 
\end{equation}
We also define a function $\textup{sgn} \colon \R \to \{\pm 1\}$ by
$\textup{sgn}(x) = 1$ for $x<0$ and $\textup{sgn}(x) = -1$ for $x
\geq 0$. 

We now give the definition of twisted cubes. 

\begin{definition}\label{definition:twisted cube}
Let $n, \mfc, \mfl$ and $A_j$ be as above. Let $C(\mfc, \mfl)$ denote
the following subset of $\R^n$: 
\begin{equation} \label{eq:defC}
\begin{split}
  C(\mfc, \mfl) \coloneqq &\{ x = (x_1, \ldots, x_n) \in \R^n \mid \\
  &\qquad A_j(x) < x_j < 0 \textup{ or } 0 \leq
  x_j \leq A_j(x)  \quad \text{ for } 1 \leq j \leq n \}. 
\end{split}
\end{equation}
Moreover we define a density function $\rho \colon \R^n \to \R$ by 
\begin{equation}
  \label{eq:def-rho}
  \rho(x) =
  \begin{cases}
    (-1)^n \prod_{k=1}^n \textup{sgn}(x_k) & \textup{ if } x \in
    C(\mfc, \mfl), \\
   0 & \textup{ else.} 
  \end{cases}
\end{equation}
Obviously, the support $\textup{supp}(\rho)$ of the density function $\rho$ is $C(\mfc, \mfl)$. We call the pair
$(C(\mfc, \mfl), \rho)$ the \textit{twisted cube associated to $\mfc$
  and $\mfl$.} 
\end{definition} 

A twisted cube may not be combinatorially equivalent to a cube $[0,1]^n$ in the standard sense. In particular, the set $C(\mfc, \mfl)$ may be neither convex
nor closed, as the following example shows. See also the discussion in \cite[\S 2.5]{Grossberg-Karshon}.

\begin{example}
  Let $n=2$, $\mfl = (\ell_1 = 2, \ell_2 = 3)$ and $\mfc =
  \{c_{12} = 1\}$. Then $C \coloneqq C(\mfc, \mfl)$ is a subset of $\R^2$ consists of points $(x_1,x_2)$ satisfying
\[
\begin{split}
& 0 \leq x_2 \leq 3, \\
2 - x_2 < x_1 < 0 \text{ or }&  0 \leq x_1 \leq 2 - x_2.
\end{split}
\]
See Figure~\ref{figure:first example} for the set $C$. The value of the density function $\rho$ is recorded within each region.

\setlength{\unitlength}{0.5cm}
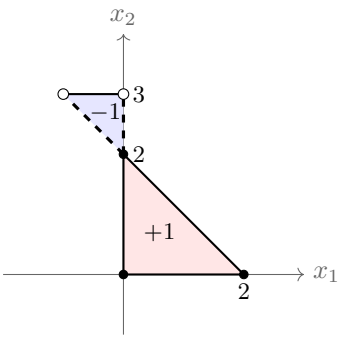
\begin{figure}[h]
\centering
\begin{tikzpicture}[scale = 0.8]
\draw[->, color=black!60!white] (-2,0)--(3,0) node [right] {$x_1$};
\draw[->, color=black!60!white] (0,-1)--(0,4) node [above] {$x_2$};

\fill[fill=blue!10!white] (0,2)--(0,3)--(-1,3)--cycle;
\filldraw[fill=red!10!white, thick] (0,0)--(2,0)--(0,2)--cycle;

\draw[very thick, dashed] (0,2)--(0,3);
\draw[very thick, dashed] (0,2)--(-1,3);
\draw[thick] (-1,3)--(0,3);

\draw[fill=black] (2,0) circle (2pt) ;
\draw[fill=black] (0,0) circle (2pt);
\draw[fill=black] (0,2) circle (2pt);
\draw[fill=white] (0,3) circle (2.5pt);
\draw[fill=white] (-1,3) circle (2.5pt);

\node at (0.6,0.7) {\small $+1$};
\node at (-0.3, 2.7) {\small $-1$};

\node[below] at (2,0) {\small $2$};
\node[right] at (0,2) {\small $2$};
\node[right] at (0,3) {\small $3$};
\end{tikzpicture}
\captionof{figure}{A twisted cube.}\label{figure:first example}

\end{figure}

Note that the subset $C$ does \emph{not} contain the points $\{ (0, x_2) \mid 2 <
x_2 < 3 \}$ and the points $\{ (x_1, x_2) \mid  2 < x_2 < 3
\textup{ and } x_1 =2  - x_2 \}$, so $C$ is neither closed nor convex.
\end{example}

As mentioned in the introduction, the primary goal of this manuscript is to give necessary and sufficient conditions for the \emph{untwistedness} of the twisted cube, in terms of the combinatorics of the defining parameters. The following makes this notion precise. 

\begin{definition}[{cf. \cite[Definition 2.2]{HaradaYang15}}]\label{definition:untwisted} 
We say that 
Grossberg--Karshon twisted cube $(C=C(\mfc, \mfl), \rho)$ is
\textit{untwisted} if 
$C$ is a closed convex
polytope, and the density function $\rho$ is
constant and equal to $1$ on $C$ and $0$ elsewhere. 
\end{definition}

The main result of \cite{HaradaYang15} characterizes the untwistedness of the Grossberg--Karshon twisted cube in terms of the basepoint-freeness of a certain toric divisor on a toric variety constructed from the data of $\mfc$ and $\mfl$.
In particular, their result can be stated in terms of the Cartier data $\{m_{\sigma}\}$ associated to the divisor on the toric variety. Before reviewing the relevant result from~\cite{HaradaYang15} we introduce some terminology.

Let $\{e_1^+,\dots,e_n^+\}$ be the standard basis of $\R^n$. 
For $\sigma = (\sigma_1,\dots,\sigma_n) \in%
\{+,-\}^n$, define $m_{\sigma} = (m_{\sigma,1}, \ldots, m_{\sigma,n}) = \sum_{j=1}^n m_{\sigma,j} e_j^+ \in \Z^n$ as 
follows. 
\begin{equation}\label{eq:def m_sigma}
m_{\sigma,j} = 
\begin{cases}
0 & \textrm{ if } \sigma_j = +, \\
A_j(m_{\sigma,j+1},\dots,m_{\sigma,n}) & \textrm{ if } \sigma_j = -.
\end{cases}
\end{equation}
Here, $A_{j}(x)$ are the functions in~\eqref{eq:def-A}. 
With this notation, we recall the following.

\begin{theorem}[{cf. \cite[Proposition 2.1]{HaradaYang15}}]\label{thm-HY} 
  Let $n, \mfc$ and $\mfl$ be as above and let $(C(\mfc, \mfl),\rho)$ denote the
  corresponding Grossberg--Karshon twisted cube. 
Then $(C(\mfc, \mfl), \rho)$ is untwisted if and only if 
$m_{\sigma, j} \geq 0$ for all $\sigma \in \{+,-\}^n$ and for 
all $j$ with $1 \leq j \leq n$.  
\end{theorem}

 \begin{remark}[{cf. \cite[Definition 2.1, Lemma 2.2]{HaradaYang15}}]
	Define
	\[
	e_j^- \coloneqq -e_j^+ - \sum_{k >j} c_{jk}e_k^+  \textrm{ for } 1 \leq j \leq n.
	\]
	Let $\Sigma(\mfc)$ be the fan consisting of maximal cones 
	generated by $\{ e_{j}^{\sigma_j} \mid 1 \leq j \leq n \}$ for each $\sigma = (\sigma_1,\dots,\sigma_n) \in \{+,-\}^n$. Then the toric variety $X(\Sigma(\mfc))$ constructed by the fan $\Sigma(\mfc)$ is called a \defi{Bott manifold}.
	For the torus-invariant divisor 
	\[
	D(\mfc, \mfl) \coloneqq  \sum_{j = 1}^r \ell_j D_{e_j^-},
	\]
	the Cartier data for $D(\mfc, \mfl)$ is same as $\{m_{\sigma}\}$. Here, $D_{e_j^-}$ be the toric-invariant divisor corresponding to the 
	ray spanned by $e_j^-$ for $1 \leq j \leq n $. 
\end{remark}

Recall that we will study the case when the
defining parameters for the Grossberg--Karshon twisted cube arise
from certain representation-theoretic data. We now briefly describe
how to derive the $\mfc$ and $\mfl$ in this case. 

Following the setting in~\cite{Grossberg-Karshon}, let $G$ be a complex semisimple
linear algebraic group
of rank $r$ over $\C$. 
Choose a Cartan subgroup $H\subset G$, and let $\mathfrak{g} = \mathfrak{h} \oplus \bigoplus_{\alpha} \mathfrak{g}_{\alpha}$ be the decomposition into root spaces.
We choose a set $\Delta^+$ of positive roots, and let $B$ be the Borel subgroup whose Lie algebra is $\mathfrak{h} \oplus \bigoplus_{\alpha \in \Delta^+} \mathfrak{g}_{-\alpha}$.
Let $\{\al_1,\ldots, \al_r\}$ denote the simple roots,
$\{\al_1^{\vee}, \ldots, \al_r^{\vee}\}$ the coroots, 
and $\{\varpi_1, \ldots, \varpi_r\}$ the fundamental weights. Note that fundamental weights are characterized by the relation $\langle\varpi_i,\al_j^{\vee}\rangle=\de_{ij}$. Let $W$ be the Weyl group of $G$ and $s_{\alpha} \in
W$ denote the simple reflection in $W$ corresponding to the
root $\alpha$. For simplicity, we denote $s_i$ for the reflection $s_{\alpha_i}$ corresponding to the simple root $\alpha_i$.

Let $\mathbf i = (i_1,\dots,i_n)$ be a sequence of elements in $[r] $ and $\mfm = (m_1,\dots,m_n) \in \Z_{\geq 0}^n$.
Then $\mathbf i$ corresponds to a decomposition of an element $w = s_{i_1} s_{i_2} \cdots s_{i_n}$ in $W$ which is not necessarily reduced. 
For such $\mathbf i$ and $\mfm$, we define constants $\mfc(\mathbf i)  = \{c_{jk}\}_{1 \leq j < k \leq n}$ and $\mfl(\mathbf i, \mfm)  = (\ell_1,\dots,\ell_n)$ by the formulas in~\cite[\S3.7]{Grossberg-Karshon}
\begin{equation}\label{eq:def cjk rep}
c_{jk}=\lee \al_{i_k}, \al_{i_j}^{\vee}\ree \quad \text{ for }1 \leq j < k \leq n,
\end{equation}
\begin{equation}\label{eq:def ellj rep}
\ell_j = \langle m_j \varpi_{i_j} + \cdots + m_n \varpi_{i_n}, \alpha_{i_j}^{\vee} \rangle \quad
\text{ for }1 \leq j \leq n.
\end{equation}
Note that the constants $c_{jk}$ are Cartan integers of $G$.  
The following example illustrates these definitions.

\begin{example}\label{example_TC_integers_121}
  Consider $G = \SL(3,\C)$ with simple roots $\{\alpha_1,
  \alpha_2\}$. Let $\mathbf i = (1,2,1)$ and $\mfm = (1,1,1)$. Then we have 
  \[
    c_{12} = \langle \alpha_2, \alpha_1^{\vee} \rangle = -1, \quad
    c_{13} = \langle \alpha_1, \alpha_1^{\vee} \rangle = 2, \quad
    c_{23}  = \langle \alpha_1, \alpha_2^{\vee} \rangle = -1, 
  \]
  and 
  \[
  \begin{split}
  \ell_1 &= \langle \varpi_1 + \varpi_2 + \varpi_1, \alpha_1^{\vee} \rangle = 2, \\
  \ell_2 &= \langle \varpi_2 + \varpi_1, \alpha_2^{\vee} \rangle = 1,\\
  \ell_3 &= \langle \varpi_1 + \alpha_1^{\vee} \rangle = 1. 
  \end{split}
  \]
\end{example}

\begin{example}
	Consider $G = \SO(8)$ with simple roots $\{\alpha_1,\dots,\alpha_4\}$. See Table~\ref{Dynkin} for the numbering on simple roots. Let $\mathbf i = (1,2,3,2,4)$ and $\mfm = (2,1,3,1,1)$. Then  the integers $c_{jk}$ and $\mfl$ are given by
	\[
	(c_{jk})
	= \begin{pmatrix}
	0 & -1 & 0 & -1 & 0 \\
	0 & 0 & -1 & 2 & -1 \\
	0 & 0 & 0 & -1 & 0 \\
	0 & 0 & 0 & 0 & -1 \\
	0 & 0 & 0 & 0 & 0
	\end{pmatrix},
	\quad 	\mfl = (2, 2, 3, 1, 1).
	\]
\end{example}

Geometrically, a word $\mathbf i = (i_1,\dots,i_n)$ and the integer vector $\mfm$ defines a \defi{Bott--Samelson variety} $Z_{\mathbf i}$ and a line bundle $\mathcal L_{\mathbf i, \mfm}$ on it. More precisely, the Bott--Samelson variety $Z_{\mathbf i}$ is defined to be the quotient
\[
Z_{\mathbf i} = (P_{i_1} \times \cdots \times P_{i_n})/B^n,
\]
where $P_i$ is the parabolic subgroup associated with the simple roots $\alpha_i$, i.e., its Lie algebra is $\mathfrak{g}_{\alpha_i} \oplus \mathfrak{b}$, and the right action of $B^n$ on $P_{i_1} \times \cdots \times P_{i_n}$ is given by
\[
(p_1,\dots,p_n) \cdot (b_1,\dots,b_n) = (p_1b_1,b_1^{-1}p_2b_2,\dots,b_{n-1}^{-1} p_nb_n).
\]
The multiplication map $(p_1,\dots,p_n) \mapsto p_1 \cdots p_n$ yields a well-defined morphism
\[
\mu \colon Z_{\mathbf i} \to G/B.
\]
The integer vector $\mfm = (m_1,\dots,m_n)$ defines the line bundle $\mathcal L_{\mathbf i, \mfm} \to Z_{\mathbf i}$ 
\[
\mathcal L_{\mathbf i, \mfm} = (P_{i_1} \times \cdots \times P_{i_n} \times \mathbb{C}) / B^n,
\]
where the right action of $B^n$ on $P_{i_1} \times \cdots \times P_{i_n} \times \mathbb{C}$ is defined by
\[
\begin{split}
&(p_1,\dots,p_n,v) \cdot (b_1,\dots,b_n) \\
&\qquad = (p_1b_1,b_1^{-1}p_2b_2,\dots,b_{n-1}^{-1}p_nb_n,
(m_1\varpi_{i_1})(b_1) \cdots (m_n \varpi_{i_n})(b_n)v).
\end{split}
\]
In this setting, the Bott--Samelson variety $Z_{\mathbf i}$ determined by the sequence $\mathbf i$ has a toric degeneration to the toric variety $X(\Sigma(\mfc))$ (see~\cite[\S 2]{Pasquier} and~\cite{Grossberg-Karshon}).
Moreover, the  line bundle $\mathcal L_{\mathbf i, \mfm}$ over the Bott--Samelson variety $Z_{\bf i}$  degenerates into the line bundle over the Bott manifold $X(\Sigma(\mfc))$. In particular, the degeneration of the line bundle $\mathcal L_{\mathbf i, \mfm}$  over $X(\Sigma(\mfc))$ is given by the divisor $D(\mfc, \mfl)$. 

The set $H^0(Z_{\mathbf i}, \mathcal{L}_{\mathbf i, \mfm})$ of holomorphic sections possesses a $B$-representation structure and is called a \defi{generalized Demazure module}. Moreover, the ordinary Demazure module can be obtained in this way.
Indeed, suppose that $\mathbf i = (i_1,\dots,i_n) \in [r]^n$ is a reduced decomposition of an element $w$ in the Weyl group $W$, i.e., $w = s_{i_1} \cdots s_{i_n}$ and $\ell(w) = n$. Then a dominant integral weight $\lambda = \sum_{i=1}^r \lambda_i \varpi_i$\footnote{A weight $\lambda = \sum_{i=1}^r \lambda_i \varpi_i$ is dominant integral  if $\lambda_i \in \Z_{\geq 0}$ for all $i$.} gives a line bundle~$\mathcal{L}_{\lambda} $ on the flag variety $G/B$, so that on the Schubert variety $X(w) \coloneqq \overline{BwB/B} \subset G/B$. Define $\mfm = (m_1,\dots,m_n) \in \Z_{\geq 0}^n$ by
\begin{equation}\label{eq_m_j_Demazure}
m_j = 
\begin{cases}
\lambda_{i_j} & \text{ if } i_{j'} \neq i_j \text{ for any } j < j' < n, \\
0 &\text{ otherwise.}
\end{cases}
\end{equation}
for $1 \leq j \leq n$. Then, we have that
\begin{equation}\label{eq_relation_L_lambda_im}
\mu^{\ast} \mathcal{L}_{\lambda} = \mathcal{L}_{\mathbf i, \mfm},
\end{equation}
and the morphism $\mu$ induces an isomorphism of $B$-modules
\begin{equation}\label{eq_Demazure_module_equivalence}
H^0(X(w), \mathcal{L}_{\lambda}) \simeq H^0(Z_{\mathbf i}, \mu^{\ast} \mathcal{L}_{\lambda})
\simeq \mathbb{C}_{-\lambda'} \otimes H^0(Z_{\mathbf i}, \mathcal{L}_{\mathbf i, \mfm})
\end{equation}
where $\lambda' = \sum_{i \in [n] \setminus \{i_1,\dots,i_n\}} \lambda_i \varpi_i$.
We note that the relation~\eqref{eq_relation_L_lambda_im} and the second isomorphism of $B$-modules in~\eqref{eq_Demazure_module_equivalence} holds even when $\mathbf i$ is not reduced.
See~\cite[Section~2]{Fujita} and reference therein for more details on generalized Demazure modules.

\begin{remark}\label{rmk_gDemazure_module}
	Since the Picard number of $Z_{\mathbf i}$ is $n$ and that of $G/B$ is $r$, not every line bundle over $Z_{\mathbf i}$ can be obtained from $\mathcal{L}_{\lambda}$ usually. 
	For example, let $G = \SL(3,\C)$ and $\mathbf i = (1,2,1)$. Consider a morphism 
	\(
	\mu \colon Z_{\mathbf i} \to G/B
	\)
	and a dominant weight $\lambda = \lambda_1 \varpi_1 + \lambda_2 \varpi_2$. Using~\eqref{eq_relation_L_lambda_im}, we obtain
	\[
	\mu^{\ast} \mathcal{L}_{\lambda} = \mathcal{L}_{(1,2,1), (0, \lambda_2, \lambda_1)}. 
	\]
	For this reason, any line bundle $\mathcal{L}_{(1,2,1),(m_1,m_2,m_3)}$ with $m_1\neq 0$ cannot be expressed as the form $\mu^{\ast} \mathcal{L}_{\lambda}$. 
	
\end{remark}

As a consequence of equations~\eqref{eq:def ellj rep} and~\eqref{eq_m_j_Demazure}, we obtain the following result which will be used later:
\begin{lemma}\label{lemma_ellj_and_lambda_ij}
	Let $\mathbf i = (i_1,\dots,i_n) \in [r]^n$ and $\lambda = \sum_{i=1}^r \lambda_i \varpi_i$ 
	a dominant weight. Suppose that $\mfm$ is given by~\eqref{eq_m_j_Demazure}.
	 Then the constant $\mfl(\mathbf i, \mfm) = (\ell_1,\dots,\ell_n)$ is given by the formula
	 \[
	 \ell_j = \lambda_{i_j} \quad \text{ for }1 \leq j \leq n. 
	 \]
\end{lemma}
\begin{proof}
	From~\eqref{eq:def ellj rep} and~\eqref{eq_m_j_Demazure}, we get
	\[
	\begin{split}
	\ell_j & = \langle m_j \varpi_{i_j} + \cdots + m_n \varpi_{i_n}, \alpha_{i_j}^{\vee} \rangle 
	\quad (\because~\eqref{eq:def ellj rep}) \\
	&= \sum_{\substack{j' \geq j, \\ i_j = i_{j'}}} m_{j'} \\
	&= \lambda_{i_j} \quad (\because~\eqref{eq_m_j_Demazure}).
	\end{split}
	\]
	This proves the lemma.
\end{proof}

As mentioned in the introduction,  Grossberg and Karshon derive a Demazure-type character formula for the  $B$-representation $H^0(Z_{\mathbf i}, \mathcal{L}_{\mathbf i, \mfm})$ corresponding to $\mathbf i$ and $\mfm$, expressed as a sum over the lattice points $\Z^n \cap C(\mfc, \mfl)$ in the Grossberg--Karshon twisted cube $(C(\mfc, \mfl),\rho)$ (see \cite[Theorems 5 and 6]{Grossberg-Karshon}). The lattice points appear with a plus or minus sign according the density function~$\rho$. Accordingly, their formula 
is a \emph{positive} formula if $\rho$ is constant and equal to $1$
 on all
of $C(\mfc,\mfl)$. 
From the point of view of representation theory it
is therefore of interest to determine conditions on the integer vector $\mfl = (\ell_1,\dots,\ell_n)$ and 
the word decomposition $\mathbf i = (i_1,\dots,i_n)$ such that the associated
Grossberg--Karshon twisted cube is in fact untwisted.

\section{Diagram jumping walks, hesitant jumping walk avoidance, and \\
	the statement of the main theorem}\label{sec:walks}

In order to state our main theorem, it is useful to introduce some
terminology. 
From now on, we assume that the group $G$ is simply-laced, i.e., the Dynkin diagram of $G$ only contains simple links, so that $G$ is of type $A$, $D$, or $E$. 
In what follows, we fix an ordering on the simple roots as
in Table~\ref{Dynkin}; our conventions agree with that in the
standard textbook of Humphreys~\cite{Humphreys}. 
\begin{center}
	\begin{table}[b]
  \begin{tabular}{c|c  }
  \toprule
  $\Phi$ & Dynkin diagram \\
  \midrule
  \raisebox{1em}{$A_r$ $(r \geq 1)$} &
     \begin{tikzpicture}[scale=.5]
       \foreach \x in {0,...,4}
       \draw[xshift=\x cm] (\x cm,0) circle (.1cm);
       \draw[dotted] (4.1 cm,0) -- +(1.8 cm,0);
       \foreach \y in {0.05,1.05,3.05,...,4.05}
       \draw[xshift=\y cm] (\y cm,0) -- +(1.8 cm,0);
       \node [below] at (0,-.2) {\small{$1$}};
       \node [below] at (2,-.2) {\small{$2$}};
       \node [below] at (4,-.2) {\small{$3$}};
       \node [below] at (6,-.2) {\small{$r-1$}};
       \node [below] at (8,-.2) {\small{$r$}};
     \end{tikzpicture}  \\[0.5em]

  \raisebox{1.5em}{$D_r$ $(r \geq 4)$} & 
   \begin{tikzpicture}[scale=.5]
\foreach \x in {0,...,3}
\draw[xshift=\x cm,] (\x cm,0) circle (.1cm);
\draw[xshift=6 cm] (30: 17 mm) circle (.1cm);
\draw[xshift=6 cm] (-30: 17 mm) circle (.1cm);
\draw[dotted] (2.1 cm,0) -- +(1.8 cm,0);
\foreach \y in {0.05, 2.05,...,3.05}
\draw[xshift=\y cm] (\y cm,0) -- +(1.8 cm,0);
\draw[xshift=6 cm] (30: 1 mm) -- (30: 16 mm);
\draw[xshift=6 cm] (-30: 1 mm) -- (-30: 16 mm);
\node [below] at (0,-.2) {\small{$1$}};
\node [below] at (2,-.2) {\small{$2$}};
\node [below] at (4,-.2) {\small{$r-3$}};
\node [below] at (6,-.2) {\small{$r-2$}};
\node [right] at (7.8,.7) {\small{$r-1$}};
\node [right] at (7.6,-1) {\small{$r$}};
\end{tikzpicture}
  \\[0.5em]
  \raisebox{1.5em}{$E_6$} & 
 \begin{tikzpicture}[scale=.5]
\foreach \x in {0,...,4}
\draw[,xshift=\x cm] (\x cm,0) circle (1 mm);
\foreach \y in {0,...,3}
\draw[,xshift=\y cm] (\y cm,0) ++(.1 cm, 0) -- +(18 mm,0);
\draw[] (4 cm,2 cm) circle (1 mm);
\draw[] (4 cm, 1mm) -- +(0, 1.8 cm);
\node [below] at (0,-.2) {\small{$1$}};
\node [below] at (2,-.2) {\small{$3$}};
\node [below] at (4,-.2) {\small{$4$}};    
\node [below] at (6,-.2) {\small{$5$}};
\node [below] at (8,-.2) {\small{$6$}};    
\node [right] at (4,2) {\small{$2$}};
\end{tikzpicture}

  \\
  \raisebox{1.5em}{$E_7$} & 
   \begin{tikzpicture}[scale=.5]
\foreach \x in {0,...,5}
\draw[,xshift=\x cm] (\x cm,0) circle (1 mm);
\foreach \y in {0,...,4}
\draw[,xshift=\y cm] (\y cm,0) ++(.1 cm, 0) -- +(18 mm,0);
\draw[] (4 cm,2 cm) circle (1 mm);
\draw[] (4 cm, 1mm) -- +(0, 1.8 cm);
\node [below] at (0,-.2) {\small{$1$}};
\node [below] at (2,-.2) {\small{$3$}};
\node [below] at (4,-.2) {\small{$4$}};    
\node [below] at (6,-.2) {\small{$5$}};
\node [below] at (8,-.2) {\small{$6$}};    
\node [below] at (10,-.2) {\small{$7$}};    
\node [right] at (4,2) {\small{$2$}};
\end{tikzpicture}

  \\
  \raisebox{1.5em}{$E_8$} & 
  \begin{tikzpicture}[scale=.5]
\foreach \x in {0,...,6}
\draw[,xshift=\x cm] (\x cm,0) circle (1 mm);
\foreach \y in {0,...,5}
\draw[,xshift=\y cm] (\y cm,0) ++(.1 cm, 0) -- +(18 mm,0);
\draw[] (4 cm,2 cm) circle (1 mm);
\draw[] (4 cm, 1mm) -- +(0, 1.8 cm);
\node [below] at (0,-.2) {\small{$1$}};
\node [below] at (2,-.2) {\small{$3$}};
\node [below] at (4,-.2) {\small{$4$}};    
\node [below] at (6,-.2) {\small{$5$}};
\node [below] at (8,-.2) {\small{$6$}};    
\node [below] at (10,-.2) {\small{$7$}};    
\node [below] at (12,-.2) {\small{$8$}};    
\node [right] at (4,2) {\small{$2$}};
\end{tikzpicture}
  \\
  \bottomrule
\end{tabular}
\captionof{table}{Dynkin diagrams.}\label{Dynkin}
\end{table}
\end{center}

In order to simplify the notation, we define \textit{distance} $d(A,B)$ of two subsets $A, B \subset [r]$ to be 
\begin{equation}\label{eq_def_distance}
d(A,B) \coloneqq \min \{ d(a,b)\mid a \in A, b \in B\}.
\end{equation}
Here, $d(a,b)$ for $a,b \in [r]$ is the minimal distance of elements $a,b$ in the corresponding Dynkin diagram.
For example, suppose that $G$ is of type $A_5$. Then	we have the following enumerations.
	\[
	d(\{1,2,3\}, \{4,5\}) = 1, \quad d(\{1,2\}, \{2\}) = 0.
	\] 
Moreover, when $G$ is of type $D_4$, then $d(\{1,2\},\{4\}) = 1$ and $d(\{3\},\{4\})=2$.
\begin{definition}
  Let $ \mathbf i = (i_1, i_2, \ldots, i_n) \in [r]^n$.
  We say that $\mathbf i$ is a \textit{jumping walk}
  if 
  \[
  d(i_{j}, \{i_1,\dots,i_{j-1}\}) = 1
  \]
  for all $1 \leq j \leq n$.
\end{definition}

\begin{example}\label{example:diagram walks}
	\begin{enumerate}
		\item In Type $A$, the words $(1,2,3,4)$, $(3,2,1,4,5)$, and $(3,2,4,1,5)$ are all jumping walks. See Figures~\ref{fig_jumping_walk_1}, \ref{fig_jumping_walk_2}, \ref{fig_jumping_walk_3}.
		\item In Type $D$, $(r-1,r-2,r-3,r,r-4)$ is a jumping walk. See Figure~\ref{fig_jumping_walk_4} for $r = 5$.
		\item In Type $E_8$, $(4,2,3,5,1,6,7,8)$ is a jumping walk. See Figure~\ref{fig_jumping_walk_5} 
		\end{enumerate}
\end{example}
\begin{figure}
\begin{center}
\begin{subfigure}[b]{0.25\textwidth}
	\centering
	     \begin{tikzpicture}[scale=.43]
	\foreach \x in {0,...,3}
	\draw[xshift=\x cm] (\x cm,0) circle (.1cm) node (\x) {};
	\foreach \y in {0.05,1.05,2.05}
	\draw[xshift=\y cm] (\y cm,0) -- +(1.8 cm,0);
	\node [below] at (0,-.2) {\small{$1$}};
	\node [below] at (2,-.2) {\small{$2$}};
	\node [below] at (4,-.2) {\small{$3$}};
	\node [below] at (6,-.2) {\small{$4$}};

	\path[->, red, thick] (0) edge [bend left] (1) {}
	 (1) edge [bend left] (2) {}
	  (2) edge [bend left] (3) {};
	
	\end{tikzpicture}
	\caption{$(1,2,3,4)$.}\label{fig_jumping_walk_1}
\end{subfigure}	
\begin{subfigure}[b]{0.3\textwidth}
	\centering
	\begin{tikzpicture}[scale=.43]
	\foreach \x in {0,...,4}
	\draw[xshift=\x cm] (\x cm,0) circle (.1cm) node (\x) {};
	\foreach \y in {0.05,...,3.05}
	\draw[xshift=\y cm] (\y cm,0) -- +(1.8 cm,0);
	\node [below] at (0,-.2) {\small{$1$}};
	\node [below] at (2,-.2) {\small{$2$}};
	\node [below] at (4,-.2) {\small{$3$}};
	\node [below] at (6,-.2) {\small{$4$}};
	\node [below] at (8,-.2) {\small{$5$}};
	
	\path[->, red, thick] (2) edge [bend right] (1) {}
	(1) edge [bend right] (0) {}
	(0) edge [bend left] (3) {}
	(3) edge [bend left] (4) {};
	
	\end{tikzpicture}
	\caption{$(3,2,1,4,5)$.}\label{fig_jumping_walk_2}
\end{subfigure}	
\begin{subfigure}[b]{0.3\textwidth}
	\centering
	\begin{tikzpicture}[scale=.43]
	\foreach \x in {0,...,4}
	\draw[xshift=\x cm] (\x cm,0) circle (.1cm) node (\x) {};
	\foreach \y in {0.05,...,3.05}
	\draw[xshift=\y cm] (\y cm,0) -- +(1.8 cm,0);
	\node [below] at (0,-.2) {\small{$1$}};
	\node [below] at (2,-.2) {\small{$2$}};
	\node [below] at (4,-.2) {\small{$3$}};
	\node [below] at (6,-.2) {\small{$4$}};
	\node [below] at (8,-.2) {\small{$5$}};
	
	\path[->, red, thick] (2) edge [bend right] (1) {}
	(1) edge [bend left] (3) {}
	(3) edge [bend right] (0) {}
	(0) edge [bend left] (4) {};
	
	\end{tikzpicture}
	\caption{$(3,2,4,1,5)$.}\label{fig_jumping_walk_3}
	\end{subfigure}

\hspace{1em}

\begin{subfigure}[b]{0.3\textwidth}
   \begin{tikzpicture}[scale=.43]
\foreach \x in {1,...,3}
\draw[xshift=\x cm,] (\x cm,0) circle (.1cm) node (\x) {};
\draw[xshift=6 cm] (30: 17 mm) circle (.1cm) node (4) {};
\draw[xshift=6 cm] (-30: 17 mm) circle (.1cm) node (5) {};
\foreach \y in {1.05, 2.05}
\draw[xshift=\y cm] (\y cm,0) -- +(1.8 cm,0);
\draw[xshift=6 cm] (30: 1 mm) -- (30: 16 mm);
\draw[xshift=6 cm] (-30: 1 mm) -- (-30: 16 mm);
\node [below] at (2,-.2) {\small{$1$}};
\node [below] at (4,-.2) {\small{$2$}};
\node [below] at (6,-.2) {\small{$3$}};
\node [right] at (7.8,.7) {\small{$4$}};
\node [right] at (7.6,-1) {\small{$5$}};

	\path[->, red, thick] (4) edge [bend right] (3) {}
(3) edge [bend left] (2) {}
(2) edge [bend right] (5) {}
(5) edge [bend left] (1) {};
\end{tikzpicture}
\caption{$(4,3,2,5,1)$.}\label{fig_jumping_walk_4}
\end{subfigure}
\begin{subfigure}[b]{0.3\textwidth}
  \begin{tikzpicture}[scale=.43]
\foreach \x in {0,...,6}
\draw[,xshift=\x cm] (\x cm,0) circle (1 mm) node (\x) {};
\foreach \y in {0,...,5}
\draw[,xshift=\y cm] (\y cm,0) ++(.1 cm, 0) -- +(18 mm,0);
\draw[] (4 cm,2 cm) circle (1 mm) node (2') {} ;
\draw[] (4 cm, 1mm) -- +(0, 1.8 cm);
\node [below] at (0,-.2) {\small{$1$}};
\node [below] at (2,-.2) {\small{$3$}};
\node [below] at (4,-.2) {\small{$4$}};    
\node [below] at (6,-.2) {\small{$5$}};
\node [below] at (8,-.2) {\small{$6$}};    
\node [below] at (10,-.2) {\small{$7$}};    
\node [below] at (12,-.2) {\small{$8$}};    
\node [right] at (4,2) {\small{$2$}};

	\path[->, red, thick] (2) edge [bend right] (2') {}
(2') edge [bend right] (1) {}
(1) edge [bend right] (3) {}
(3) edge [bend left] (0) {}
(0) edge [bend right] (4) {}
(4) edge [bend right] (5) {}
(5) edge [bend right] (6) {};

\end{tikzpicture}
\caption{$(4,2,3,5,1,6,7,8)$.}\label{fig_jumping_walk_5}
\end{subfigure}
\end{center}
	\caption{Jumping walks.}\label{fig_example_jumping_walks}
\end{figure}
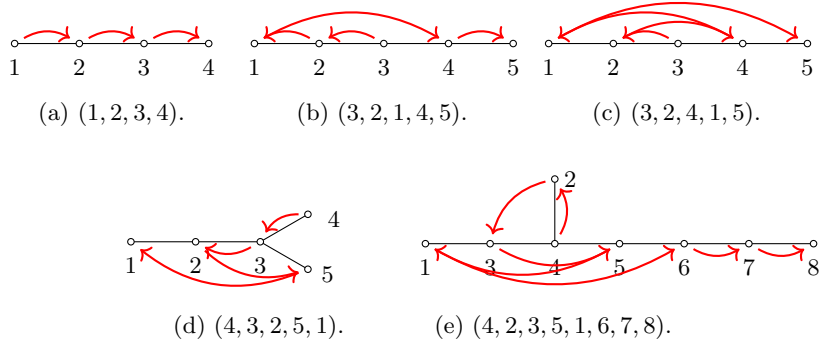

Because of the definition, a jumping walk is minimal. More precisely, the indices $\{i_1,\dots,i_n\}$ are all distinct, i.e., the jumping walk visits any given vertex of the Dynkin diagram at most once.

In what follows, we also find it useful to consider words which are
`almost' jumping walks, except that the word begins with a repetition
(thus disqualifying it from being a walk), i.e. the initial index
appears twice. 

\begin{definition}
  Let $\mathbf{i}= (i_0,i_1,  \ldots, i_n)\in [r]^{n+1}$.  We say that
  $\mathbf{i}$ is a \textit{hesitant jumping walk} if 
  \begin{itemize}
  \item $n \geq 1$, 
  \item $i_0 = i_1$, and 
  \item the subword $(i_1, \ldots, i_n)$ is a jumping walk. 
  \end{itemize}
  In other words, except for the `hesitation' at the first step, the
  remainder of the word is a jumping walk. We refer to the subword $(i_1, \ldots,
  i_n)$ as the \textit{jumping component} of the hesitant jumping walk. 
\end{definition}

\begin{definition}
	Let $\ell = (\ell_0, \ell_1,\dots,\ell_n) \in \Z_{\geq 0}^{n+1}$ and $\mathbf{i} \in [r]^{n+1}$. We say that $\mathbf i$ is a \textit{hesitant jumping $\mfl$-walk} if
	\begin{itemize}
		\item $\mathbf i$ is a hesitant jumping walk, and
		\item  $\ell_{0} - \ell_{1} < \ell_{1} + \ell_{2} + \cdots + \ell_{n}$.
	\end{itemize} 
A word $\mathbf i$ is \textit{hesitant-jumping-$\mfl$-walk-avoiding} if there is no subword $\mathbf{j} = (i_{j_0},i_{j_1},\dots,i_{j_s})$ of $\mathbf{i}$ which is a hesitant jumping $(\ell_{j_0},\ell_{j_1},\dots,\ell_{j_s})$-walk. 
\end{definition}

\begin{example}\label{example_121_hesitant_jumping_subword}
Let $G = \SL(3,\C)$. Suppose that $\mathbf i = (1,2,1)$ and $(\ell_1,\ell_2,\ell_3) \in \Z_{\geq 0}^3$. Then the indices $(j_0,j_1) = (1,3)$ defines a hesitant jumping walk 
\[
(i_{j_0},i_{j_1}) = (i_1,i_3) = (1,1).
\] 
Consequently, the word $\mathbf i$ is hesitant-jumping-$\mfl$-walk-avoiding if and only if the integers $\ell_1,\ell_2,\ell_3$ satisfies $\ell_1 - \ell_3 \geq \ell_3$. 
\end{example}

\begin{example}
	Let $G = \SL(4,\C)$. Suppose that $\mathbf i = (1,2,1,3,2,1)$ and $(\ell_1,\dots,\ell_6) \in \Z_{\geq 0}^6$. Then there are five hesitant jumping subwords of $\mathbf i$.
	\[
	\begin{array}{l|cccccc}
	(j_0,j_1,\dots,j_s) & 1 & 2 & 1 & 3 & 2 & 1 \\
	\hline
 	(1,3) & 1 & & 1 \\
 	(1,3,5) & 1 & & 1 & & 2 \\
 	(2,5) & & 2 & & & 2 \\
 	(2,5,6) & & 2 & & & 2 & 1 \\
 	(3,6) & & & 1 &  & & 1
	\end{array}
	\]
	Consequently, the word $\mathbf i$ is hesitant-jumping-$\mfl$-walk-avoiding if and only if the integers $\ell_1,\dots,\ell_6$ satisfies
	\[
	\begin{split}
	\ell_1 -  \ell_3 &\geq \ell_3, \quad
	\ell_1 - \ell_3 \geq \ell_3 + \ell_5, \\
	\ell_2 - \ell_5 &\geq \ell_5,  \quad
	\ell_2 - \ell_5 \geq \ell_5 + \ell_6, \quad
	\ell_3 - \ell_6 \geq \ell_6.
	\end{split}
	\]
\end{example}
Given the terminology introduced above we  state our main theorem.
\begin{theorem}\label{theorem:main}
	Let $G$ be a complex simply-connected semisimple algebraic group of rank $r$.
	Let $\mathbf i = (i_1, i_2, \ldots, i_n) \in [r]^n$ be a word  and let 
	$\mfl = (\ell_1, \ldots, \ell_n) \in \Z_{\geq 0}^n$.
	Let $\mfc = \{c_{jk}\}$  be determined from $\mathbf i$ as in~\eqref{eq:def cjk rep}.
	Then the corresponding Grossberg--Karshon twisted cube $(C(\mfc, \mfl),
	\rho)$ is untwisted if and only if $\mathbf i$ is hesitant-jumping-$\mfl$-walk-avoiding.
\end{theorem}

\begin{example}
	Let $G = \SL(3,\C)$ and $\mathbf i = (1,2,1)$. Then by Example~\ref{example_TC_integers_121}, the subset $C(\mathbf c, \ell) \subset \R^3$ consists of points $(x_1,x_2,x_3)$ satisfying:
	\[
\begin{split}
\ell_1 + x_2 - 2 x_3 < x_1 < 0 \quad &\text{ or } \quad 0 \leq x_1 \leq \ell_1 + x_2 - 2 x_3, \\
\ell_2 + x_3 < x_2 < 0\quad &\text{ or } \quad 0 \leq x_2 \leq \ell_2 + x_3, \\
\ell_3 < x_3 < 0 \quad & \text{ or }\quad 0 \leq x_3 \leq \ell_3.
\end{split}
	\]
	By Example~\ref{example_121_hesitant_jumping_subword}, the word $\mathbf i$ is hesitant-jumping-$\mfl$-walk-avoiding if and only if 
	\[
	\ell_1 - \ell_3 \geq \ell_3.
	\]
	In Figure~\ref{fig_TC_121_example}, we draw the twisted cubes for $(\ell_1,\ell_2,\ell_3) = (3,1,1)$ and $(2,1,2)$, and the former one gives untwisted twisted cube but not the latter. We present the lattice points in twisted cubes with plus sign in red color and with minus sign in blue color.
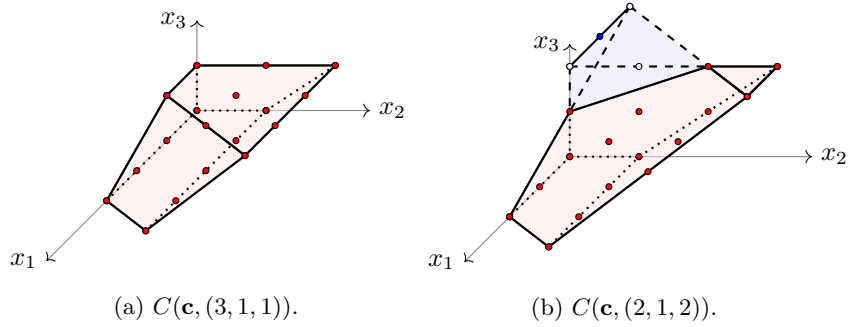
\begin{figure}
	\begin{subfigure}{0.43\textwidth}
		\centering
		\begin{tikzpicture}[x=2.3cm, y=1.5cm, z=-1cm, scale=.4]
		
		\coordinate (1) at (2,1,3);
		\coordinate (2) at (2,1,0);
		\coordinate (3) at (0,1,1);
		\coordinate (4) at (0,1,0);
		\coordinate (5) at (1,0,4);
	 	\coordinate (6) at (1,0,0);
	 	\coordinate (7) at (0,0,3);
	 	\coordinate (8) at (0,0,0);
\draw [->, draw=gray] (0,0,0) -- (2.5,0,0) node [right] {$x_2$};
\draw [->, draw=gray] (0,0,0) -- (0,2,0) node [left] {$x_3$};
\draw [->, draw=gray] (0,0,0) -- (0,0,5) node [left] {$x_1$};

	\coordinate (x1) at (0,0,0);
	\coordinate (x2) at (1,0,0);
	\coordinate (x3) at (0,1,0);
	\coordinate (x4) at (0,1,1);
	
	\coordinate (x6) at (1,0,1);
	\coordinate (x7) at (1,0,2);
	\coordinate (x8) at (1,0,3);
	\coordinate (x9) at (1,0,4);
	
	\coordinate (x10) at (0,0,1);
	\coordinate (x11) at (0,0,2);
	\coordinate (x12) at (0,0,3);
	
	\coordinate (x13) at (1,1,0);
	\coordinate (x14) at (2,1,0);
	
	\coordinate (x15) at (2,1,1);
	\coordinate (x16) at (2,1,2);
	\coordinate (x17) at (2,1,3);
	\coordinate (x18) at (1,1,1);
	\coordinate (x5) at (1,1,2);


	 	\draw[thick,fill=red!10!white,semitransparent] (2)--(4)--(3)--(1)--cycle;
	 	\draw[thick] (2)--(4)--(3)--(1)--cycle;
	 	\draw[thick,fill=red!10!white,semitransparent] (1)--(5)--(7)--(3);
	 	\draw[thick ] (1)--(5)--(7)--(3);
	 	\draw[thick, dotted](2)--(6)--(5);
	 	\draw[thick, dotted] (7)--(8)--(4);
	 	\draw[thick, dotted] (8)--(6);
	 	\draw[thick] (1)--(3);

	 	\foreach \x in {1,...,8}	\draw[fill=black] (\x) circle (.1cm);
	 	
	\foreach \y in {1,...,18} \draw[fill=red] (x\y) circle (.1cm);

\end{tikzpicture}
\caption{ $C(\mathbf c, (3,1,1))$.}
\end{subfigure}
	\begin{subfigure}{0.43\textwidth}
		\centering
	\begin{tikzpicture}[x=2.3cm, y=1.5cm, z=-1cm, scale=.4]
	
	\coordinate (1) at (3,2,1);
	\coordinate (2) at (3,2,0);
	\coordinate (3) at (0,2,-2);
	\coordinate (4) at (0,2,0);
	\coordinate (5) at (1,0,3);
	\coordinate (6) at (1,0,0);
	\coordinate (7) at (0,0,2);
	\coordinate (8) at (0,0,0);
	
	\coordinate (m1) at (2,2,0);
	\coordinate (m2) at (0,1,0);

	\draw [->, draw=gray] (0,0,0) -- (3.5,0,0) node [right] {$x_2$};
\draw [->, draw=gray] (0,0,0) -- (0,2.5,0) node [left] {$x_3$};
\draw [->, draw=gray] (0,0,0) -- (0,0,3.5) node [left] {$x_1$};

\coordinate (x1) at (0,0,0);
\coordinate (x2) at (0,0,1);

\coordinate (x4) at (1,0,0);
\coordinate (x5) at (1,0,1);
\coordinate (x3) at (1,0,2);
\coordinate (x7) at (3,2,1);
\coordinate (x8) at (3,2,0);
\coordinate (x9) at (0,0,2);
\coordinate (x10) at (1,0,3);

\coordinate (x11) at (1,1,0);
\coordinate (x6) at (1,1,1);

\coordinate (x12) at (2,1,2);
\coordinate (x13) at (2,1,1);
\coordinate (x14) at (2,1,0);

\coordinate (y1) at (0,2,-1);

	\draw[fill=blue!10!white,draw=none,semitransparent] (3)--(4)--(m1)--cycle;
	\draw[fill=blue!10!white, draw=none,semitransparent] (4)--(m2)--(m1)--cycle;
	
	\draw[thick, fill=red!10!white, semitransparent]  (m1)--(m2)--(7)--(5)--(1)--cycle;
	\draw[thick, fill=red!10!white,semitransparent] (1)--(2)--(m1);
	
		\draw[thick]  (m1)--(m2)--(7)--(5)--(1)--cycle;
	\draw[thick] (1)--(2)--(m1);
	

\draw[thick] (m1)--(1);
	\draw[thick] (3)--(4);
	\draw[thick] (1)--(2);
	\draw[thick, dashed] (3)--(m1);
	\draw[thick, dashed] (4)--(m1);
	\draw[thick, dashed] (3)--(m2);
	\draw[thick, dashed] (4)--(m2);
	
	\draw[thick, dotted] (m2)--(8)--(7);
	\draw[thick, dotted] (8)--(6)--(5);
	\draw[thick, dotted] (2)--(6);
	
	\foreach \x in {1,...,8}	\draw[fill=black] (\x) circle (.1cm);
	
 	\draw[fill=white] (3) circle (.1cm);
		\draw[fill=white] (4) circle (.1cm);
	\draw[fill =white] (1,2,0) circle (.1cm);
	\draw[fill=red] (m1)circle (.1cm);
	\draw[fill=red] (m2)circle (.1cm);
		
			\foreach \y in {1,...,14} \draw[fill=red] (x\y) circle (.1cm);
			\foreach \z in {1} \draw [fill=blue] (y\z) circle (.1cm);
	
	\end{tikzpicture}
	\caption{$C(\mathbf c, (2,1,2))$.}
\end{subfigure}
\caption{Twisted cubes for $G = \SL(3,\C)$ and $\mathbf i = (1,2,1)$. }\label{fig_TC_121_example}
\end{figure}
\end{example}

We now present a corollary of Theorem~\ref{theorem:main} which is already observed by Harada and the author~\cite{HaradaLee}. In order to state the corollary, it is useful to introduce some terminology in the paper~\cite{HaradaLee}. We call a word $\mathbf i = (i_1,\dots,i_n)$ is a \defi{diagram walk} if $d(i_j, i_{j+1}) = 1$ for all $1 \leq j < n$. For a dominant weight $\lambda = \sum_{i=1}^r \lambda_r \varpi_r$, we say $\mathbf i = (i_0,i_1,\dots,i_n)$ is a \defi{hesitant $\lambda$-walk} if $i_0 = i_1$, the subword $(i_1,\dots,i_n)$ is a diagram walk, and $\lambda_{i_n} > 0$. Lastly, we say $\mathbf i$ is \defi{hesitant-$\lambda$-walk-avoiding} if there is no subword which is a hesitant $\lambda$-walk.  With these expressions, we present the following corollary. 
\begin{corollary}[{see~\cite[Theorem~2.9]{HaradaLee}}]\label{corollary_HL}
	Let $G$ be a complex simply-connected semisimple algebraic group of rank $r$.
	Let $\mathbf i =  (i_1,i_2,\dots,i_n) \in [r]^n$ be a word and let $\lambda$ be a dominant weight. Let $\mfc = \{c_{jk}\}$ be determined from $\mathbf i$ as in~\eqref{eq:def cjk rep} and $\mfl = (\ell_1,\dots,\ell_n)$ is given by Lemma~\ref{lemma_ellj_and_lambda_ij}. Then the corresponding Grossberg--Karshon twisted cube $(C(\mfc, \mfl), \rho)$ is untwisted if and only if $\mathbf i$ is hesitant-$\lambda$-walk-avoiding.
\end{corollary}
\begin{proof}
	By Theorem~\ref{theorem:main}, it is enough to show that the word $\mathbf i$ is hesitant-jumping-$\mfl$-walk-avoiding if and only if it is hesitant-$\lambda$-walk-avoiding.  
	We prove the contrapositive of the claim, that is, we will prove that $\mathbf i$ has a subword which is a hesitant jumping $\mfl$-walk if and only if it has a subword which is a hesitant $\lambda$-walk.

	We now suppose that $\mathbf i$ has a subword $\mathbf j = (i_{j_0},i_{j_1},\dots,i_{j_s})$  which is a hesitant $\lambda$-walk. Since a diagram walk is a jumping walk, $\mathbf j$ is a hesitant jumping walk. Moreover, the condition $i_{j_0} = i_{j_1}$ implies that 
$\ell_{j_0} = \ell_{j_1}$
	 by Lemma~\ref{lemma_ellj_and_lambda_ij}, and the condition $\lambda_{i_{j_s}} > 0$ indicates that $\ell_{j_s} = \lambda_{i_{j_s}} >0$. Therefore we get 
	 \[
	 \ell_{j_0} - \ell_{j_1} = 0 < \ell_{j_1} + \cdots + \ell_{j_s},
	 \]
	 which provides that $\mathbf j$ is a hesitant jumping $\mfl$-walk.
	 
	On the other hand, suppose that $\mathbf i$ has a subword $\mathbf j = (i_{j_0},i_{j_1},\dots,i_{j_s})$  which is a hesitant jumping $\mfl$-walk. Then, the condition $i_{j_0} = i_{j_1}$ provides $\ell_{j_0} = \ell_{j_1}$ so that we get the inequality
	\[
	0 = \ell_{j_0} - \ell_{j_1} < \ell_{j_1} + \cdots + \ell_{j_s}.
	\]
	Because $\ell_j \geq 0$ for all $j$, there exists $t \in [s]$ such that $\ell_{j_t} > 0$. 
	By the definition of  jumping walk, one can always find a subword of $\mathbf j$ which is a hesitant $\lambda$-walk starting with $i_{j_0}, i_{j_1}$ and ending at $i_{j_s}$. This proves that $\mathbf i$ has a subword which is a hesitant $\lambda$-walk, so the result follows.
\end{proof}

We enclose this section presenting one more application of Theorem~\ref{theorem:main} on the Newton--Okounkov body theory. Let $\mathbf i = (i_1,\dots,i_n) \in [r]^n$ and $\mfm = (m_1,\dots,m_n) \in \Z^n_{\geq 0}$.
Let $Z_{\mathbf i}$ and $\mathcal{L}_{\mathbf i, \mfm}$ be the Bott--Samelson variety and the line bundle on it given by $\mf i$ and $\mfm$. Suppose that $\mfc = \{c_{jk}\}$ and $\mfl$ are determined from~\eqref{eq:def cjk rep} and~\eqref{eq:def ellj rep}, and $C(\mfc, \mfl)$ is the corresponding twisted cube.
Harada and Yang~\cite{Harada-Yang:2014a} construct a valuation $\nu \colon \mathbb{C}(Z) \to \Z^n$ such that if the corresponding twisted cube $C(\mfc, \mfl)$ is untwisted then it coincides with the Newton--Okounkov body $\Delta = \Delta(Z_{\mathbf i}, \mathcal{L}_{\mathbf i, \mfm}, \nu)$ defined to be
\[
\Delta = \overline{
\text{conv} \left( 
\bigcup_{k > 0} \left\{ 
\frac{\nu(\sigma)}{k}~\colon~\sigma \in H^0(Z_{\mathbf i}, \mathcal{L}_{\mathbf i, \mathbf m}^{\otimes k}) \setminus \{0\}
\right\}
\right)
}
\]
up to a certain coordinate change (see~\cite[Theorem~3.4]{Harada-Yang:2014a} for more details). As a consequence, we get the following corollary which is related to the question (2) in~\cite[Section~5]{HaradaLee}.
\begin{corollary}
	Let $G$ be a complex simply-connected semisimple algebraic group of rank $r$. Let $\mathbf i = (i_1,\dots,i_n) \in [r]^n$ be a word and let $\mfm = (m_1,\dots,m_n) \in \Z_{\geq 0}^n$. Let $\mfc = \{c_{jk}\}$ and $\mfl  \in \Z^n$ be determined from $\mathbf i$ and $\mfm$ as in~\eqref{eq:def cjk rep} and~\eqref{eq:def ellj rep}. If the word $\mf i$ is hesitant-jumping-$\mfl$-walk-avoiding, then the twisted cube $C(\mfc, \mfl)$ coincides with the Newton--Okounkov body $\Delta (Z_{\mathbf i}, \mathcal{L}_{\mathbf i, \mfm}, \nu)$ up to a coordinate change, where $\nu$ is the valuation constructed in~\cite{Harada-Yang:2014a}.
\end{corollary}

\section{Proof of the main theorem}
\label{sec:proof}

In this section, we will present a proof of the main theorem. 
We start with a proposition which will be used in the proof. One can see that the subsequent proposition holds only when simply-laced cases. 

\begin{proposition}\label{lemma_connected_hesitant}
	Let $G$ be a complex simply-connected semisimple algebraic group of rank $r$.
	Let $\mathbf i = (i_1,i_2,\dots,i_n) \in [r]^n$ be a word, and let $\mathbf{c} = \{c_{jk}\}$ be determined from $\mathbf i$ as in~\eqref{eq:def cjk rep}. 
	Suppose that $(j_0 < j_1 < j_2 < \cdots < j_s)$ is an increasing sequence of elements in $[n]$, and $\mathbf j = (i_{j_0},i_{j_1},\dots,i_{j_s})$ is the corresponding subword.
	Then, the sequence $(j_0 < j_1 < j_2 < \cdots < j_s)$ satisfies the condition
	\[
	\begin{split}
	&c_{j_0,j_1} = 2,\\
	&c_{j_1,j_t} + c_{j_2,j_t} + \cdots + c_{j_{t-1},j_t} = -1\quad \text{ for } 2 \leq t \leq s
	\end{split}
	\]
	if and only if	$\mathbf j$ is  a hesitant jumping walk.
\end{proposition}
\begin{lemma}\label{lemma_distance_and_sum_of_C_integers}
	Suppose that $G$ is  simply-laced and $(i_1,\dots,i_{n-1})$ is a jumping walk. For $i_n \in [r]$, we have the following.
	\begin{enumerate}
		\item If $d(i_n, \{i_1,\dots,i_{n-1}\}) = 0$, then 
		$\c_{i_n,i_1} +  \cdots + \c_{i_n,i_{n-1}} \geq 0$.
		\item If $d(i_n, \{i_1,\dots,i_{n-1}\}) = 1$, then
		$\c_{i_n,i_1} +  \cdots + \c_{i_n,i_{n-1}} = -1$.
		\item  If $d(i_n, \{i_1,\dots,i_{n-1}\}) > 1$, then
		$\c_{i_n,i_1} +  \cdots + \c_{i_n,i_{n-1}}  = 0$.
	\end{enumerate}
	Here, $\c_{j,k}$ are Cartan integers: 
	\begin{equation}\label{eq_Cartan_integers}
	\c_{j,k}  = \langle \alpha_j, \alpha_k^{\vee} \rangle = \begin{cases}
	2 & \text{ if } j = k, \\
	-1 & \text{ if } d(j,k) = 1, \\
	0 & \text{ otherwise}.
	\end{cases}
	\end{equation}
\end{lemma}
\begin{proof}
	Note that since $(i_1,\dots,i_{n-1})$ is a jumping walk, the set $ I \coloneqq \{i_1,\dots,i_{n-1}\}$ forms an interval on the Dynkin diagram, i.e., if $j \in [n]$ satisfies $\min I < j < \max I$,
	then $j \in I$. 
	
	We consider the first case. Since the distance between $i_n$ and the set $I$ is zero, $i_n \in I$. Suppose that $i_n$ sits in the right-most/left-most position among $I$, i.e., $i_n = \min I$ or $i_n = \max I$. Then,
	\[
	\c_{i_n,i_1} + \cdots + \c_{i_n,i_{n-1}} = 2 -1 = 1 \geq 0,
	\]
	which is the desired inequality.
	Suppose that $i_n$ is neither maximum nor minimum of $I$. Then, we have that
	\[
	\c_{i_n,i_1} +  \cdots + \c_{i_n,i_{n-1}} = 2 -1 -1  = 0 \geq 0.
	\]	
	Consequently, we prove the claim for (1). 
	For the second and the third cases, by the definition of Cartan integers, we get the required equalities.
\end{proof}
\begin{proof}[Proof of Proposition~\ref{lemma_connected_hesitant}]
	Suppose that a subword $\mathbf j= (i_{j_0},i_{j_1},\dots,i_{j_s})$ of $\mathbf i$ is a hesitant jumping walk. Then, by the hesitant condition $i_{j_0} = i_{j_1}$, we have $c_{j_0,j_1} = 2$. Moreover, the jumping walk condition $d(i_{j_t}, \{i_{j_1},\dots,i_{j_{t-1}}\}) = 1$ and Lemma~\ref{lemma_distance_and_sum_of_C_integers} implies that
	\[
	c_{j_1,j_t} + \dots + c_{j_{t-1},j_t}
	= \c_{i_{j_t},i_{j_1}} + \dots + \c_{i_{j_t},i_{j_{t-1}}} 
	= -1,
	\] 
	which proves the ``if'' part of the proposition.
	
	Suppose that we have an increasing sequence $(j_0 < j_1 < \cdots < j_s)$ satisfies conditions. Consider the subword $\mathbf  j = (i_{j_0},i_{j_1},\dots,i_{j_s})$.  The first condition $c_{j_0,j_1} = 2$ implies that $i_{j_0} = i_{j_1}$. Accordingly, the word $\mathbf j$ hesitates at the first. When $t = 2$, the second condition becomes $c_{j_1,j_2} = -1$. Then by~\eqref{eq_Cartan_integers}, we have that $d(i_{j_2},i_{j_1}) = 1$, so that $(i_{j_0},i_{j_1},i_{j_2})$ is a hesitant jumping walk. Using an induction on $s$ and Lemma~\ref{lemma_distance_and_sum_of_C_integers}, we prove the ``only if'' part of the proposition. 
\end{proof}

\subsection{Necessity}
\label{sec:necessity}

We first prove that if $\mathbf i$ has a subword which is a hesitant jumping $\mfl$-walk, then the corresponding twisted cube is twisted. 
Suppose that $\mathbf j= (i_{j_0}, i_{j_1},\dots,i_{j_s})$ is a subword of $\mathbf i$ which is a  hesitant jumping walk. Also suppose that 
\begin{equation}\label{eq_condition_on_ell}
\ell_{j_0} - \ell_{j_1} < \ell_{j_1} + \cdots + \ell_{j_s}.
\end{equation}
Then by Proposition~\ref{lemma_connected_hesitant}, the integers $\{c_{jk}\}$ satisfying that
\begin{eqnarray}
&& c_{j_0,j_1} = 2, \label{eq_c_positive} \\
&& c_{j_1,j_t} + c_{j_2,j_t} + \cdots + c_{j_{t-1},j_t} = -1 \quad \text{ for } 2 \leq t \leq s. \label{eq_sum_c_negative} 
\end{eqnarray}
We then wish to show that $(C(\mathbf{c},\ell), \rho)$ is twisted. 
To prove that $(C(\mathbf{c},\ell),\rho)$ is twisted, it is enough find an element $\sigma$ of $\{+,-\}^n$ and an index $k \in [n]$ such that $m_{\sigma, k} < 0$.
To achieve this, we consider the element $\sigma = (\sigma_1,\dots,\sigma_n) \in \{+,-\}^n$ given by
\[
\sigma_p = \begin{cases}
-& \text{ if } p \in \{j_0,j_1,\dots,j_s\}, \\
+ & \text{ otherwise.}
\end{cases}
\]
By the definition of $m_{\sigma}$ and $m_{\sigma, p} = 0$ for $p \notin \{j_0,j_1,\dots,j_s\}$, we then have
\begin{equation}
m_{\sigma,j_t} = \ell_{j_t} - \sum_{ p \in \{j_{t+1},\dots,j_s\}} c_{j_t,p} m_{\sigma,p} \quad \text{ for } 1 \leq t \leq s. 
\end{equation}
We know that $c_{j_0,j_1} = \langle \alpha_{i_{j_0}}, \alpha_{i_{j_1}}^{\vee} \rangle = 2$ if and only if ${i_{j_0}} = {i_{j_1}}$.
Moreover, in this case we have $c_{j_0, p} = c_{j_1,p}$ for all $p$. 
From these considerations, we have:
\[
\begin{split}
m_{\sigma,j_0} &= \ell_{j_0} - \sum_{p \in \{j_1,\dots,j_s\}} c_{j_0,p} m_{\sigma, p} \\
&= \ell_{j_0} - c_{j_0,j_1} m_{\sigma, {j_1}} - \sum_{p \in \{j_2,\dots,j_s\}} c_{j_0,p} m_{\sigma,p} \\
&= \ell_{j_0} - 2\left(\ell_{j_1} - \sum_{p \in \{j_2,\dots,j_s\}} c_{j_1,p} m_{\sigma,p} \right) - \sum_{p \in \{j_2,\dots,j_s \}} c_{j_0,p} m_{\sigma,p} \\
&= \ell_{j_0} - 2\ell_{j_1} + \sum_{p \in \{j_2,\dots,j_s\}} c_{j_1,p} m_{\sigma_p}\\
&= \ell_{j_0} - \ell_{j_1} - m_{\sigma, j_1}
\end{split}
\]

We now claim that $m_{\sigma, j_0} < 0$; as already noted, this suffices to prove the theorem. In order to prove this claim, it is enough to show that
\begin{equation}\label{eq_m_j1_equality}
m_{\sigma, j_1} = \ell_{j_1} + \cdots + \ell_{j_s}.
\end{equation}
This is because if the inequality~\eqref{eq_m_j1_equality} holds, then we get
\[
\begin{split}
m_{\sigma,j_0} &= \ell_{j_0} - \ell_{j_1} - m_{\sigma,j_1} \\
&= \ell_{j_0} - \ell_{j_1} - (\ell_{j_1} + \cdots +\ell_{j_s}) \\
&< 0 \quad (\because~\eqref{eq_condition_on_ell})
\end{split}
\]
which proves the claim, so now we prove~\eqref{eq_m_j1_equality}. Using~\eqref{eq_sum_c_negative} and the definition of $m_{\sigma}$, we have that
\[
\begin{split}
m_{\sigma,j_1} &= \ell_{j_1}  - c_{j_1,j_2} m_{\sigma,j_2} - \cdots - c_{j_1,j_s} m_{\sigma, j_s} \\
&= \ell_{j_1} + (\ell_{j_2} - c_{j_2,j_3} m_{\sigma,j_3} - \cdots - c_{j_2,j_s} m_{\sigma,j_s}) - c_{j_1,j_3} m_{\sigma,j_3} - \cdots - c_{j_1,j_s} m_{\sigma,j_s} \\
&= \ell_{j_1} + \ell_{j_2} - (c_{j_2,j_3} + c_{j_1,j_3})m_{\sigma, j_3} - \cdots - (c_{j_2,j_s} + c_{j_1,j_s}) m_{\sigma,j_s} \\
&= \ell_{j_1} + \ell_{j_2} + (\ell_{j_3} - c_{j_3,j_4} m_{\sigma,j_4} - \cdots - c_{j_3,j_s} m_{\sigma, j_s}) - \cdots - (c_{j_2,j_s} + c_{j_1,j_s}) m_{\sigma,j_s} \\
&= \ell_{j_1} + \ell_{j_2} + \ell_{j_3} - (c_{j_3,j_4} + c_{j_2,j_4} + c_{j_1,j_4}) m_{\sigma,j_4} - \cdots - (c_{j_3,j_s}+c_{j_2,j_s} + c_{j_1,j_s}) m_{\sigma,j_s} \\
& = \cdots = \ell_{j_1} + \ell_{j_2} + \cdots + \ell_{j_{s-1}} - (c_{j_{s-1},j_s} + \cdots + c_{j_1,j_s}) m_{\sigma,j_s} \\
& = \ell_{j_1} + \ell_{j_2} + \cdots + \ell_{j_s}.  
\end{split}
\]
Consequently, we prove the equation~\eqref{eq_m_j1_equality}, so the necessity of the theorem follows.

\subsection{Sufficiency}
\label{sec:sufficiency}

We now prove  that twistedness implies the existence of a subword which is a hesitant jumping $\mfl$-walk. To give a proof, we prepare one lemma.

\begin{lemma}\label{lemma_sum_c_negative}
	Suppose that a sequence $(j_1 < j_2 < \cdots < j_{s-1})$ of indices defines a jumping walk $(i_{j_1},\dots,i_{j_{s-1}})$.
	If for $j_s > j_{s-1}$ we have
	\[
	c_{j_1, j_s} + c_{j_2, j_s} + \cdots + c_{j_{s-1},j_s} < 0,
	\]
	then 
	$(j_1 < j_2 < \cdots < j_{s-1} < j_s)$ also defines a jumping walk and
	$c_{j_1, j_t} + c_{j_2, j_t} + \cdots + c_{j_{t-1},j_t} =-1$.
\end{lemma}
\begin{proof}
	Assume on the contrary that $(j_1 < j_2 < \cdots < j_{s-1} < j_s)$ does not define a jumping walk, i.e., $d(i_{j_s}, \{i_{j_1},\dots,i_{j_{s-1}}\}) \neq 1$. Then, by Lemma~\ref{lemma_distance_and_sum_of_C_integers}, we have that
	\[
	c_{j_1, j_s} + c_{j_2, j_s} + \cdots + c_{j_{s-1},j_s}  = 
	\c_{i_{j_s}, i_{j_1}} + \c_{i_{j_s},i_{j_2}} + \cdots + \c_{i_{j_s}, i_{j_{s-1}}} \geq 0,
	\]
	which contradicts to the assumption. As a consequence, we prove the lemma. 
\end{proof}

By Theorem~\ref{thm-HY}, there exists an element $\sigma$ of $\{+,-\}^n$ and an index $k$ such that $m_{\sigma, k} < 0$. For such a choice of $\sigma$, we may assume without loss of generality that $k$ is chosen to be the maximal such index, i.e.,  $m_{\sigma, k} < 0$ and $m_{\sigma, s} \geq 0$ for $s > k$. Recall that
\[
m_{\sigma, k} = \ell_k - \sum_{s > k}c_{ks} m_{\sigma, s}.
\]
By assumption $m_{\sigma, k} < 0$, we have that 
\[
\sum_{s > k} c_{ks} m_{\sigma, s} > \ell_k \geq 0.
\] 
Since $m_{\sigma, s} \geq 0$ for $s > k$, this implies that there exists some $p >k $ with $c_{k,p} > 0$ and $m_{\sigma, p} > 0$. Choose $j_1$ to be the minimal such index. Consequently, $c_{k,s} \leq 0$ or $m_{\sigma, s} = 0$ for all $k < s < j_1$, so we have that
\begin{equation}\label{eq_proof_1}
\begin{split}
\ell_k < \sum_{s > k} c_{ks} m_{\sigma, s} \leq c_{k,j_1} m_{\sigma, j_1} + \sum_{s > j_1} c_{k,s} m_{\sigma, s}
\end{split}
\end{equation}
By definition, we have that $c_{k,j_1} = \langle \alpha_{i_{j_1}}, \alpha_{i_{k}}^{\vee} \rangle >0$ if and only if ${i_k} = {i_{j_1}}$. Furthermore, in this case we get $c_{k,j_1} = 2$ and $c_{j_1,s} = c_{k,s}$ for all $s$. From these observations, we get:
\begin{equation}\label{eq_proof_2}
c_{k,j_1} m_{\sigma,j_1} +\sum_{s > j_1} c_{k,s} m_{\sigma,s}  = 2 \left(\ell_{j_1} - \sum_{s > j_1} c_{j_1,s} m_{\sigma,s} \right) + \sum_{s > j_1} c_{j_1,s} m_{\sigma,s}.
\end{equation}
Combining~\eqref{eq_proof_1} and~\eqref{eq_proof_2}, we have that
\begin{equation}\label{eq_proof_3}
\ell_k - \ell_{j_1} < \ell_{j_1} - \sum_{s> j_1} c_{j_1,s} m_{\sigma, s} = m_{\sigma, j_1}.
\end{equation}
First suppose $- \sum_{s > j_1} c_{j_1,s} m_{\sigma, s} \leq 0$. In this case, we have that
\[
\ell_k - \ell_{j_1} < \ell_{j_1}.
\]
Accordingly, the sequence $(j_0 = k < j_1)$ satisfies the three required conditions of hesitant jumping $\mfl$-walk, so we are done. 

On the other hand, if $- \sum_{s >j_1} c_{j_1,s} m_{\sigma, s} > 0$, we set $j_0 = k$ and define $j_2$ as follows. Since $m_{\sigma, s} \geq 0$ for $s > k$ by assumption, in order for  the summand $\sum_{s > j_1} c_{j_1,s} m_{\sigma, s} $ to be strictly negative there must exist an index $j_2 > j_1$ with $c_{j_1 j_2} < 0$ and $m_{\sigma, j_2} >0$. Note that since $c_{j_1 j_2} = \langle \alpha_{i_{j_2}}, \alpha_{i_{j_1}}^{\vee} \rangle$ we have that $c_{j_1 j_2} = -1$. Choose $j_2$ to be the minimal such index, i.e., $c_{j_1, s} \geq 0$ or $m_{\sigma, s} = 0$ for all $j_1 < s < j_2$. Then we have that
\begin{equation}\label{eq_proof_4}
\begin{split}
m_{\sigma, j_1} &= \ell_{j_1} - \sum_{s > j_1} c_{j_1,s} m_{\sigma, s} \\
&\leq \ell_{j_1} - c_{j_1 j_2} m_{\sigma, j_2} - \sum_{s > j_2} c_{j_1,s} m_{\sigma, s}  \\
&=\ell_{j_1} + m_{\sigma, j_2} - \sum_{s > j_2} c_{j_1,s} m_{\sigma,s} \\
&= \ell_{j_1} + \ell_{j_2} - \sum_{s > j_2} (c_{j_2,s} + c_{j_1,s}) m_{\sigma,s}
\end{split}
\end{equation}
If $- \sum_{s > j_2} (c_{j_2,s} + c_{j_1,s}) m_{\sigma,s} \leq 0$, then the sequence $(j_0=k < j_1 < j_2)$ satisfies the three required conditions since we get the following inequality by considering~\eqref{eq_proof_3} and~\eqref{eq_proof_4}:
\[
\ell_{j_0}- \ell_{j_1} < m_{\sigma,j_1} \leq \ell_{j_1} + \ell_{j_2}.
\]
Otherwise, i.e., $ \sum_{s > j_2} (c_{j_2,s} + c_{j_1,s}) m_{\sigma,s} < 0$, then we choose $j_3>j_2$ to be the minimal index such that $c_{j_2,s} + c_{j_1,s} < 0$ and $m_{\sigma,s} > 0$. Since $(j_1 < j_2)$ defines a jumping walk, we have that $(j_1 < j_2 < j_3)$ defines a jumping walk and $c_{j_2,j_3} + c_{j_1,j_3} = -1$ by Lemma~\ref{lemma_sum_c_negative}. Therefore the inequality~\ref{eq_proof_4} becomes to
\begin{equation}
\begin{split}
m_{\sigma,j_1} &\leq \ell_{j_1} + \ell_{j_2} + m_{\sigma, j_3} - \sum_{s > j_3}(c_{j_2,s} + c_{j_1,s}) m_{\sigma,s} \\
&= \ell_{j_1} + \ell_{j_2} + \ell_{j_3} - \sum_{s > j_3}(c_{j_3,s} + c_{j_2,s} + c_{j_1,s}) m_{\sigma,s}
\end{split}
\end{equation}
If $- \sum_{s > j_3}(c_{j_3,s} + c_{j_2,s} + c_{j_1,s}) m_{\sigma,s} \leq 0$, then the sequence $(j_0 = k < j_1 < j_2 < j_3)$ satisfies the three required conditions. Otherwise, we may repeat the above argument as many times as necessary. Since the indices $j_t$ are bounded above by~$n$, this process must stop, i.e., there must exist some $s \geq 1$ such that the sequence $j_0 < j_1 < \cdots < j_s$ found in this manner satisfies the requirements. 
Consequently, we prove the sufficiency of the theorem.


\end{document}